\documentclass[11pt]{article}
\usepackage{amsthm, amsmath, amssymb, amsfonts, url, booktabs, tikz, setspace, fancyhdr, bm}
\usepackage{geometry}
\usepackage{hyperref, enumerate}
\usepackage[capitalise]{cleveref}
\usepackage{comment}
\usetikzlibrary{patterns}
\usetikzlibrary{decorations.pathmorphing,decorations.pathreplacing,decorations.shapes}

\newtheorem{theorem}{Theorem}[section]
\newtheorem{claim}[theorem]{Claim}
\newtheorem{lemma}[theorem]{Lemma}
\newtheorem{fact}[theorem]{Fact}
\newtheorem{corollary}[theorem]{Corollary}

\newtheorem{proposition}[theorem]{Proposition}
\newtheorem*{remark*}{Remark}

\newtheorem{question}[theorem]{Question}

\def\eps{\varepsilon}

\newenvironment{poc}{\begin{proof}[Proof of claim]}{\end{proof}}

\def\R{\mathbb R}
\def\N{\mathbb N}
\def\Z{\mathbb Z}
\def\C{\mathbb C}

\renewcommand{\Re}{\operatorname{Re}}
\renewcommand{\Im}{\operatorname{Im}}

\usepackage{todonotes}

\title{Graph with any rational density and no rich subsets of linear size}
\author{Seonghyuk Im \thanks{Department of Mathematical Sciences, KAIST, South Korea and Extremal Combinatorics and Probability Group (ECOPRO), Institute for Basic Science (IBS), Daejeon, South Korea.  Supported by the National Research Foundation of Korea (NRF) grant funded by the Korea government(MSIT) No. RS-2023-00210430 and by the Institute for Basic Science (IBS-R029-C4). E-mail: seonghyuk@kaist.ac.kr.}
\and Suyun Jiang\thanks{School of Artificial Intelligence, Jianghan University, Wuhan, Hubei, China, and Extremal Combinatorics and Probability Group (ECOPRO), Institute for Basic Science (IBS), Daejeon, South Korea. Supported by National Natural Science Foundation of China (11901246) and China Scholarship Council and IBS-R029-C4.
Email: jiang.suyun@163.com. }
\and
Hong Liu \thanks{Extremal Combinatorics and Probability Group (ECOPRO), Institute for Basic Science (IBS), Daejeon, South Korea.
Supported by IBS-R029-C4.
 Email: hongliu@ibs.re.kr. }
\and 
Tuan Tran \thanks{School of Mathematical Sciences, University of Science and Technology of China, China.  Supported by National Key Research and Development Program of China 2023YFA 1010200, and the Excellent Young Talents Program (Overseas) of the National Natural Science Foundation of China. Email: trantuan@ustc.edu.cn.}
}

\begin{document}
\maketitle
\begin{abstract}
    A well-known application of the dependent random choice asserts that any $n$-vertex graph $G$ with positive edge density  
    contains a `rich' vertex subset $U$ of size $n^{1-o(1)}$ such that every pair of vertices in $U$ has at least $n^{1-o(1)}$ common neighbors.
    In 2003, using a beautiful construction on hypercube, Kostochka and Sudakov showed that this is tight: one cannot 
    remove the $o(1)$ terms even if the edge density of $G$ is  
    $1/2$. 
    In this paper, we generalize their result from pairs to tuples. To be precise, we show that given every pair of positive integers $p<q$, there is an $n$-vertex graph $G$ for all sufficiently large $n$ with edge density $p/q$ such that any vertex subset $U$ of size $\Omega(n)$ contains $q$ vertices, any $p+1$ of which have $o(n)$ common neighbors. The edge density $p/q$ is best possible. Our construction uses isoperimetry and concentration of measure on high dimensional complex spheres. 
\end{abstract}

\section{Introduction}
{\em Dependent random choice} is a simple yet surprisingly powerful probabilistic technique which has many striking applications to Extremal Graph Theory, Ramsey Theory, Additive Combinatorics, and Combinatorial Geometry.
Its early version was  
proved and applied by various researchers, starting with Gowers \cite{Gowers1998new}, Kostochka and R\"{o}dl \cite{Kostochka2001graphs}, and Sudakov \cite{Sudakov2003few}. 
For more information about dependent random choice and its applications, we refer interested readers to the survey \cite{fox2011dependentrandomchoice}. 
The following lemma~\cite{Alon2003turannumber, Kostochka2001graphs, Sudakov2003few} is a typical result proved by dependent random choice. Let $a, d, m, n, r$ be positive integers. Let $G = (V, E)$ be a graph with $|V| = n$ and average degree $d = 2|E(G)|/n$. If there is a positive integer $t$ such that $\frac{d^t}{n^{t-1}}-\binom{n}{r}\left(\frac{m}{n}\right)^t \ge a$,
then $G$ contains a subset $U$ of at least $a$ vertices such that every $r$ vertices in $U$ have at least $m$ common neighbors.  
A useful corollary of this lemma reads as follows:
\begin{itemize}
    \item[] \emph{Every $n$-vertex graph with $\Omega(n^2)$ edges contains a vertex subset $U$ of size $n^{1-o(1)}$ in which every two vertices have $n^{1-o(1)}$ common neighbors.} 
\end{itemize}
In many applications, it would be ideal to have the size of $U$ and the number of common neighbors of all its pairs be linear in $n$.
This problem was raised in the study of the famous Burr-Erd\H{o}s conjecture~\cite{Burr1973, Kostochka2004}. 
However, Kostochka and Sudakov~\cite{kostochka_sudakov_2003} gave a negative answer via a surprising construction on hypercube.

\begin{theorem}[\cite{fox2011dependentrandomchoice,kostochka_sudakov_2003}]\label{thm:pair}
    For sufficiently large $n$, there is a graph $G$ on $n$ vertices with edge density $1/2+o(1)$ such that any subset $U \subseteq V(G)$ of linear size
    contains a pair of vertices with at most $o(n)$ common neighbors.
\end{theorem}

In this paper, we construct graphs with any rational edge density, generalizing \cref{thm:pair} from pairs to tuples.

\begin{theorem}\label{thm:tuple}
Let $p$ and $q$ be positive integers with $p<q$. Then for sufficiently large $n$, there is a graph $G$ on $n$ vertices with edge density $p/q+o(1)$ such that any subset $U \subseteq V(G)$ of size $\Omega(n)$ contains $q$ vertices, any $p+1$ of which have at most $o(n)$ common neighbors.
\end{theorem}

The edge density $p/q$ in \cref{thm:tuple} is best possible, as shown by the following result.

\begin{proposition}\label{prop:tightness}
    The following holds for every real $\eps>0$ and integers $p,q, n$ with $1\le p<q$ and $n\ge 3q/\eps$. 
    If $G$ is an $n$-vertex graph with 
    $e(G) \ge (\frac{p}{q} + \eps){n \choose 2}$, then there is a subset $U \subseteq V(G)$ of size 
    $|U|\ge \frac{\eps}{3}n$ such that 
    any $q$-subset  
    of $U$ contains  
    $p+1$ elements having 
    at least $\frac{\eps}{4\binom{q}{p+1}}n$ common neighbors.
\end{proposition}

\begin{proof} 
    Let $X$ be the set of vertices of degree less than 
    $(\frac{p}{q} +\frac{\eps}{2})n$, and 
    $U$ be the complement of $X$. Then 
    $\sum_{v\in X}d_{G}(v)<(\frac{p}{q}+\frac{\eps}{2})n^2$. It follows that the sum of the degrees of vertices in 
    $U$ is at least 
    \[
    2e(G)-\sum_{v\in X}d_{G}(v) \ge 2\left(\frac{p}{q} + \eps\right){n \choose 2}-\left(\frac{p}{q}+\frac{\eps}{2}\right)n^2\geq \frac{\eps}{3}n^2,
    \]
    which implies $|U|\ge \frac{\eps}{3}n$. 
    Consider any $q$-subset $U'$ of $U$.
    The sum of the degree of vertices in $U'$ is at least 
    $|U'|\cdot (\frac{p}{q}+\frac{\eps}{2})n=(p+q\eps/2)n$.
    Thus, by the pigeonhole principle, there are at least $\eps n/4$ vertices in $V(G) \setminus U'$  
    each of which has at least 
    $p+1$ neighbors in 
    $U'$.
    Again by the pigeonhole principle, there are $p+1$ vertices in $U'$ having at least $\frac{\eps n}{4\binom{q}{p+1}}$ common neighbors. 
\end{proof}

\cref{thm:tuple} is a consequence of a more general Ramsey-type result we prove for set-colorings. A $(q,p)$-coloring of the complete graph $K_n$ is an assignment $\varphi$ of a set of $p$ colors from a palette of $q$ colors to each edge of $K_n$. We say that a subgraph $H$ of $K_n$ is monochromatic (with respect to $\varphi$) if $\bigcap_{e\in E(H)}\varphi(e)\ne\varnothing$. Note that when $p=1$, a $(q,1)$-coloring is just a $q$-edge-coloring of $K_n$. Ramsey-type problems for $(q,p)$-colorings have been studied by many researchers for over half a century (see e.g. \cite{Alon2002,aragao2024setcoloring, Balla2023setcoloring, Conlon2022setcoloring, Conlon2023setcoloring, Erdos1965, Xu2009}). Our Ramsey-type result for $(q,p)$-colorings, of which \cref{thm:tuple} is a special case, can be stated as follows.

\begin{theorem}\label{thm:ramsey_version}
    Let $p$ and $q$ be positive integers with $p<q$, and let $\eps>0$.
    Then for sufficiently large $n$, there is a $(q,p)$-coloring $\varphi$ of $K_n$ such that the following holds.
    Every monochromatic subgraph $G$ of $K_n$ does not have a 
    vertex subset $U\subseteq  V(G)$ of size at least $\eps  n$ such that any  
    $q$ 
    vertices  
    in $U$ 
    contain $p+1$ vertices 
    having at least $\eps  n$ common neighbors in $G$.
\end{theorem}

We in fact prove a slightly stronger statement where the coloring is balanced, i.e.~each color appears roughly the same number of times, see~\cref{thm:ramsey_restated}.

Let us mention another application of \cref{thm:ramsey_version} besides \cref{thm:tuple}.
For a pair of positive integers $m>s$, we say that a graph $H$ is 
\emph{$(s,m)$-rich} if every $s$ vertices have at least $m$ common neighbors.
The following question was originally asked by Frieze and Reed (see~\cite{Kostochka2004}).
\begin{question}\label{ques:rich}
    Is there some absolute constant $\eps=\eps(s)>0$ such that the following holds?
    For every $n$-vertex graph $G$, either $G$ or its complement contains an $(s, \eps n)$-rich subgraph. 
\end{question}
Kostchka and Sudakov \cite{kostochka_sudakov_2003} actually proved that the complement of the graph in \cref{thm:pair} has similar property, thus disproving \cref{ques:rich} for $s=2$. Note that one can view the above question as asking whether every $(2,1)$-coloring of $K_n$ has large monochromatic $(2,\eps n)$-rich subgraphs.
\cref{thm:ramsey_version} 
provides a counterexample 
to a 
natural variant of this question.
\begin{corollary}
    For every $\eps>0$ and integer 
    $q\ge 2$, the following holds for sufficiently large $n$.
    There exists a  
    $(q,q-1)$-coloring of $K_n$ 
    without monochromatic 
    $(q, \eps n)$-rich subgraphs.
\end{corollary}

\paragraph{Proof sketch.}
We now sketch some of the ideas of the proof of \cref{thm:tuple}. Let us first recall the simplest $(p,q)=(1,2)$ case in \cref{thm:pair}. In this case, Kostochka and Sudakov \cite{kostochka_sudakov_2003} defined a graph $G$ on vertex set $V:=\{0, 1\}^m$ in which two vertices $\bm{x},\bm{y} \in V$ are adjacent if and only if their Hamming distance is at most $m/2$. It is easy to see that $G$ has edge density $1/2$. For any vertex subset $U$ of linear size, by isoperimetric inequality for Hamming cube, $U$ contains two nearly antipodal vectors $\bm{u}_1$ and $\bm{u}_2$, which can be shown to have $o(n)$ common neighbors. This construction seems difficult to generalize as the Hamming cube is intrinsically `binary' with a north and a south pole and can be partitioned evenly into two parts. Even for the triple case $q=3$, there is no natural way to partition the Hamming cube into three `nice' parts of almost equal size to define adjacency. To prove \cref{thm:tuple}, we work instead with high dimensional complex spheres. 

We distribute vertices `uniformly' over complex spheres and define adjacencies by their geometric relation. More specifically, we make use of the angle between two points on a complex sphere, which allows us to construct a graph of arbitrary rational edge density $p/q$ by taking edges $\bm{x}\bm{y}$ if $\arg\langle \bm{x}, \bm{y} \rangle \in [0, \frac{2\pi p}{q}]$.
This type of geometric construction is inspired by a recent generalization of Bollob\'as-Erd\H{o}s graph~\cite{Bollobas1976ramseyturan} by Liu, Reiher, Sharifzadeh and Staden~\cite{Liu2021geometric}. To show that this construction has the desired property, we first observe that every linear size vertex set must contain $q$ 
points such that every point is close to a rotation of another by a power of the primitive root $e^{2\pi i/q}$. We then use a convexity argument to show that any $p+1$ such points must have $o(n)$ common neighbors if $p/q\le 1/2$. This argument, however, fails when $p/q>1/2$ as the set $\{z \in \mathbb{C} \mid |z| \le 1, \arg z \in [0, \frac{2 \pi p}{q}]\}$ is not convex. To resolve this issue, we use a trick to reduce it to the case $\frac{p}{q} \le \frac{1}{2}$. In addition, as the inner product of complex vectors does not possess symmetry, we have to use more than one complex spheres to boost to the optimal edge density.

\paragraph{Organization.}
We collect some geometric properties of high dimensional complex spheres in \cref{sec:Preliminaries}, which will be used for our construction in \cref{sec:ramsey_version}.

\paragraph{Notation.}
We let $[n]$ denote the set $\{1,\ldots,n\}$. We use bold face lower case symbols for vectors. For a vector
$\bm{z}$ we denote by $z_j$ the value of its $j$-th coordinate. For $k\in \N$, let $\C^k$ denote the $k$-dimensional complex vector space, equipped with the standard inner product $\langle \bm{x}, \bm{y} \rangle =\sum_{j=1}^k x_j\overline{y_j}$. We write $|\bm{x}|=\sqrt{\langle \bm{x}, \bm{x} \rangle}$ for the $\ell_2$-norm of $\bm{x} \in \C^k$. 

We treat $q$ as a constant. We write $x \ll_q y$ to mean that for any $y \in  (0, 1]$ there exists $x_0 \in (0,1]$ depending on $y$ and possibly on $q$ such that for all $x \in  (0,x_0]$ the subsequent statement holds. Hierarchies with more constants are defined similarly and are to be read from the right to the left. 

\section{Preliminaries}\label{sec:Preliminaries}

In this section, we list some definitions and useful properties of inner product and high dimensional complex spheres. By definition, inner product has the following properties.

\begin{fact}\label{fact: inner product}
    For every $\bm{x},\bm{y} \in \C^k$ and $\theta \in \R$, we have
    $$\Im (\langle \bm{x}, -e^{\theta i}\bm{y} \rangle)=\Im (-e^{-\theta i}\langle \bm{x}, \bm{y} \rangle) \quad \mbox{and}\quad 
        \Im (e^{\theta i} \langle \bm{x}, \bm{y} \rangle)=-\Im (e^{-\theta i}\langle \bm{y}, \bm{x} \rangle).$$
\end{fact}

For $k\in \N$, let $\textsf{S}^{k-1}(\mathbb{R}) \subset \R^k$ denote the standard $(k-1)$-dimensional real unit sphere, and
write 
\[
\textsf{S}^{k-1}(\mathbb{C})=\Big\{\bm{z}\in \mathbb{C}^k:\sum_{j=1}^k |z_j|^2=1\Big\}
\]
for the $(k-1)$-dimensional complex unit sphere. When given a unit sphere, we will write $\lambda$ for the normalized Lebesgue measure so that the unit sphere has measure $1$. We will need the following result {\cite[Lemma 3.3]{Liu2021geometric}} which is a simple consequence of the isoperimetric inequality for spheres \cite{schmidt48isoperimetric} and lower and upper bounds on the measure of spherical caps \cite{balogh2013ramseyturan,Tkocz2012AnUB}.

\begin{lemma}
\label{coro: exist q vertices}
    Let $q \ge 2$ be an integer, 
    $0<\nu \le 1/q^2$ and $A\subseteq \textsf{S}^{k-1}(\mathbb{C})$ with $\lambda(A)>qe^{-k\nu/32}$. Then there are distinct points $\bm{a}_0, \ldots, \bm{a}_{q-1}\in A$ such that 
    \[
    |\zeta^r \bm{a}_r-\zeta^{s} \bm{a}_s|<\sqrt{\nu} \quad  \text{for all $r,s \in \Z/q\Z$},
    \]
    where $\zeta:=e^{2\pi i/q}$.
\end{lemma}

We also require the following folklore result which upper-bounds the measure of spherical strips. Its proof can be extracted from that of
\cite[Claim 3.6]{Liu2021geometric}. For completeness, we include its proof. 

\begin{lemma}
\label{lem:area_of_band}
For all $k\ge 3, \nu>0$ and $\bm{x} \in \textsf{S}^{k-1}(\mathbb{C})$, the measure of the strip 
\[
\Big\{\bm{y} \in \textsf{S}^{k-1}(\mathbb{C})\colon |\mathrm{Im}\langle \bm{x}, \bm{y} \rangle|\le \frac{\nu}{\sqrt{2k}}\Big\}
\]
is at most $3\nu$.
\end{lemma}
\begin{proof}
The statement is obviously true for $\nu>1/3$; so we assume that $\nu \le 1/3$. Define a spherical cap
\[
C:=\left\{\bm{y} \in \textsf{S}^{k-1}(\C)\colon |-\bm{x}i-\bm{y}|\le \sqrt{2}-\nu/\sqrt{2k}\right\}.
\]
By the known lower bound on the measure of spherical caps (see e.g. \cite[Lemma 2.1]{Liu2021geometric}), we have $\lambda(C)\ge \frac12-\sqrt{2}\nu$. For any $\bm{y} \in C$, using
$\nu/\sqrt{2k} \le 1/3$, we obtain
\[
2-2\,\mathrm{Im}\langle \bm{x}, \bm{y} \rangle=2-2\,\mathrm{Re}\langle -\bm{x}i, \bm{y} \rangle
=|-\bm{x}i-\bm{y}|^2 \le (\sqrt{2}-\nu/\sqrt{2k})^2<2-2\nu/\sqrt{2k},
\]
and so $\mathrm{Im}\langle \bm{x}, \bm{y} \rangle> \nu/\sqrt{2k}$. Thus, by symmetry, the measure of the strip is at most 
$1-2\lambda(C) \le 1-2\left(\frac12-\sqrt{2}\nu\right)=2\sqrt{2}\nu$.
\end{proof}

It is well-known that the real unit sphere can be partitioned into small pieces of equal measure (see e.g. \cite[Lemma 21]{feige2002optimality}). As the map
\[
\varphi\colon (x_1+iy_1,\ldots, x_k+iy_k)\mapsto (x_1, y_1, \ldots, x_k,y_k)
\]
from $\textsf{S}^{k-1}(\mathbb{C})$ to $\textsf{S}^{2k-1}(\mathbb{R})$ is an invertible isometry, we get the following lemma.

\begin{lemma}\label{coro: complex equal measure}
    There exists $C>0$ such that the following holds. Let $0<\nu<1$ and $n\ge (C/\nu)^{2k}$. Then $\textsf{S}^{k-1}(\mathbb{C})$ can be partitioned into $n$ pieces of equal measure, each of diameter at most $\nu$.
\end{lemma}

\section{Proof of Theorem~\ref{thm:ramsey_version}}\label{sec:ramsey_version}

We can view a $(q, p)$-coloring of the complete graph $K_N$ as a collection of $q$ graphs on a common vertex set of size $N$ such that each pair of vertices is an edge of exactly $p$ graphs. A monochromatic subgraph of $K_N$ is then simply a subgraph of one of these $q$ graphs.  Theorem~\ref{thm:ramsey_version} is thus equivalent to the following statement.

\begin{theorem}\label{thm:ramsey_restated}
    Let $p$ and $q$ be positive integers with $p<q$, and let $\eps>0$. Then for sufficiently large $N$, there is a collection of $q$ graphs $G_0,\ldots,G_{q-1}$ on a common vertex set $V$ of size $N$ such that each pair of vertices is an edge of exactly $p$ graphs and the following hold for each $f\in \Z/q\Z$:
    \begin{itemize}
        \item[(i)] $G_f$ has edge density $p/q\pm \eps$; 
        \item[(ii)] $G_f$ does not have a vertex subset $U$ of size at least $\eps  N$ such that among any $q$ vertices in $U$ there are $p+1$ vertices with at least $\eps  N$ common neighbors in $G_f$.
    \end{itemize}
\end{theorem}

In this section, we will prove this version of Theorem~\ref{thm:ramsey_version}.

\subsection{The case when \texorpdfstring{$\frac{p}{q} \le \frac{1}{2}$}{sparse}}\label{sec:3.1}
\paragraph{Construction 1.} We will construct $q$ graphs $G_0, \ldots, G_{q-1}$ on a common vertex set $V$ such that each pair of vertices is an edge of exactly $p$ graphs. Let\footnote{For concreteness, one can take $\eta=\frac{\eps}{132q}, t=\frac{5q}{\eps}$ and $k=10^7(\frac{q}{\eps})^3$.} 
$$0< 1/k \ll_q 1/t, \eta \ll_q \eps, \quad \mu:=\eta/\sqrt{2k},  \quad \zeta:=e^{2\pi i/q}  \quad \mbox{and}\quad n\ge (2C_{\ref{coro: complex equal measure}}/\mu)^{2k}.$$ 
Let $S_1, \ldots, S_t$ be $t$ copies of the sphere $\textsf{S}^{k-1}(\mathbb{C})$. We define an inner product of two vectors in different spheres by identifying them into a single sphere.
According to Lemma~\ref{coro: complex equal measure}, one can partition each $S_j$, $j\in[t]$, into $n$ domains $D_1^{(j)}, D_2^{(j)}, \ldots, D_{n}^{(j)}$ of equal measure $1/n$ and diameter at most $\mu/2$. Select a vector $\bm{v}_m^{(j)}$ from each $D_m^{(j)}$ so that $\langle \bm{v}_r^{(h)}, \bm{v}_s^{(j)}\rangle \neq 0$ for all $r,s \in [n]$ and $h,j\in [t]$. Let $V^{(j)}=\{\bm{v}_1^{(j)},\ldots, \bm{v}_n^{(j)}\}$ for $j\in [t]$, and $V = \bigcup_{j \in [t]} V^{(j)}$ with $|V|=N:=tn$. 

We now define $q$ graphs $G_0, \ldots, G_{q-1}$ on vertex set $V$ as follows. Given $f\in \Z/q\Z$, a cross pair $\bm{v}_r^{(h)}, \bm{v}_s^{(j)}$ with $r,s\in[n]$ and $1\le h<j\le t$ forms an edge of $G_f$ if and only if there exists $\alpha\in \left[\frac{2\pi f}{q}, \frac{2\pi (f+p)}{q}\right)$ such that
\[
e^{-i\alpha}\langle \bm{v}_r^{(h)}, \bm{v}_s^{(j)}\rangle \in [0,1].
\]
To finish the construction, for $r,s\in[n]$ and $j\in[t]$, arbitrarily assign each pair $v_r^{(j)}v_s^{(j)}$ to $p$ graphs. Notice that this construction yields an $N$-vertex graph for all large $N$ divisible by $t$. To get it work for all sufficiently large $N$, simply put $\lceil N/t\rceil$ or $\lfloor N/t\rfloor$ vertices in each $V^{(j)}$, $j\in [t]$, above and all the following arguments go through. For simplicity of presentation, we assume that $t\mid N$.

\medskip

It is clear by the construction that each pair of vertices in $V$ is an edge of exactly $p$ graphs. Note also that for each $j\in[t]$, all bipartite graphs $G_f[V^{(j)}, V\setminus V^{(j)}]$, $f\in \Z/q\Z$, are isomorphic. Indeed, the following map is an isomorphism between $G_f[V^{(j)}, V\setminus V^{(j)}]$ and $G_0[V^{(j)}, V\setminus V^{(j)}]$:
\begin{equation}\label{eq:iso}
    g(\bm{x}) = 
    \begin{cases}
    \zeta^{f}\bm{x}, & \text{if } \bm{x}\in S_h, h<j, \\
    \bm{x}, & \text{if } \bm{x}\in S_j, \\
    \zeta^{-f}\bm{x}, & \text{if } \bm{x}\in S_h, h>j.
  \end{cases}
\end{equation}
As the proportion of edges inside some $V^{(j)}$ is $o_t(1)$ and $\sum_{f=0}^{q-1}e(G_f)=p{N\choose 2}$, we can then infer that each $G_f$ has edge density $p/q+o(1)$. It remains to show that for each $f\in \Z/q\Z$, $G_f$ satisfies part (ii) of \cref{thm:ramsey_restated}.  Our proof relies on \cref{lem: exist q vertices,lem:co-degree-1} stated below.

\begin{proof}[Proof of \cref{thm:ramsey_restated} for $\frac{p}{q}\le \frac{1}{2}$]
Suppose to the contrary that there exist $f\in \Z/q\Z$ and a vertex subset $U\subseteq  V$ of size $|U|\ge \eps  tn$ such that among any $q$ vertices in $U$ there are $p+1$ vertices with at least $\eps  tn$ common neighbors in $G_f$. According to \cref{lem: exist q vertices} below, there are $j\in [t]$ and distinct vertices $\bm{u}_1, \ldots,\bm{u}_q\in U\cap V^{(j)}$ with $|\bm{u}_{r}-\zeta^{s-r} \bm{u}_{s}|\le 2\sqrt{\mu}$ for all $r,s$. By applying the map $g$ in~\eqref{eq:iso}, we may assume that $f=0$. Then by \cref{lem:co-degree-1} we find that any $p+1$ vertices in $\{\bm{u}_1, \ldots,\bm{u}_q\}$ have less than $\eps  tn$ common neighbors in $G_0$, a contradiction.
\end{proof}

In the rest of this subsection, we present the results needed in the above proof.

\begin{lemma}\label{lem: exist q vertices}
For any vertex set $U\subseteq V$ of size $|U|\ge \eps tn/2$, there exist $j\in [t]$ and distinct vertices $\bm{u}_1, \ldots,\bm{u}_q\in U\cap V^{(j)}$ satisfying
\begin{equation}\label{eq:key}
 |\bm{u}_{r}-\zeta^{s-r} \bm{u}_{s}|< 2\sqrt{\mu} \quad \text{for all $r,s\in [q]$.}  
\end{equation}
\end{lemma}
\begin{proof}
By the pigeonhole principle, there exists $j\in [t]$ such that $|U\cap V^{(j)}| \ge |U|/t \ge \eps  n/2$. Let $A=\bigcup \{D_m^{(j)}\colon \bm{v}_m^{(j)}\in U\cap V^{(j)}\}$. 
As each domain $D_m^{(j)}$ has measure $\frac1n$, $\lambda(A)=|U\cap V^{(j)}|/n \ge \eps/2$. 
Since $\eps/2> qe^{-k\mu/32}$, 
\cref{coro: exist q vertices} ensures that $A$ contains distinct points $\bm{a}_1,\ldots,\bm{a}_q$ with $|\zeta^{r}\bm{a}_r-\zeta^s \bm{a}_s|<\sqrt{\mu}$ for all $r,s\in [q]$, where $\bm{a}_1\in D_{b_1}^{(j)},\ldots,\bm{a}_q\in D_{b_q}^{(j)}$.
For all $r,s\in [q]$, by the triangle inequality, we obtain 
\begin{align*}
|\bm{v}_{b_r}^{(j)}-\zeta^{s-r} \bm{v}_{b_s}^{(j)}| &\le |\bm{v}_{b_r}^{(j)}-\bm{a}_r|+|\bm{a}_r-\zeta^{s-r} \bm{a}_s|+|\zeta^{s-r} \bm{a}_s-\zeta^{s-r}\bm{v}_{b_s}^{(j)}|\\
&=|\bm{v}_{b_r}^{(j)}-\bm{a}_r|+|\zeta^r\bm{a}_r-\zeta^{s} \bm{a}_s|+|\bm{a}_s-\bm{v}_{b_s}^{(j)}| < \sqrt{\mu}+2\cdot(\mu/2)\le 2\sqrt{\mu}.
\end{align*}
Thus, the vertices $\bm{u}_1:=\bm{v}_{b_1}^{(j)},\ldots,\bm{u}_q:=\bm{v}_{b_q}^{(j)}$ satisfy \eqref{eq:key}.
\end{proof}

In the proof of \cref{thm:ramsey_restated}, we applied the following result whose proof relies on \cref{lem:co-degree-2} stated below.

\begin{lemma}\label{lem:co-degree-1}
Let $j\in [t]$, and let $\bm{u}_1, \ldots,\bm{u}_q$ be $q$ vertices in $V^{(j)}$ satisfying \eqref{eq:key}. 
Then  
any $p+1$ vertices in $\{\bm{u}_1, \ldots,\bm{u}_q\}$ have less than $\eps  tn$ common neighbors in 
$G_0$.     
\end{lemma}
\begin{proof}
Consider any $(p+1)$-subset $A$ of $\{\bm{u}_1, \ldots,\bm{u}_q\}$. We claim that there exists a pair of vertices $\bm{u}_r,\bm{u}_s \in A$ with $s-r \in \{p,\ldots,q-p\}$. Note that $\{p,\ldots,q-p\}\neq\varnothing$ as $p/q\le 1/2$. For this, fix a vertex $\bm{u}_{r} \in A$. We may assume that every vertex $\bm{u}_s\in A\setminus \{\bm{u}_r\}$ satisfy 
$s\in \{r-p+1,\ldots,r+p-1\}\setminus\{r\}$, 
for otherwise, $\bm{u}_r$,$\bm{u}_s$ is the pair we seek. 
Thus, 
$A\setminus \{\bm{u}_r\} \subseteq \{\bm{u}_{r-p+1},\ldots,\bm{u}_{r+p-1}\} \setminus \{\bm{u}_r\}$. 
We partition $\{\bm{u}_{r-p+1},\ldots,\bm{u}_{r+p-1}\} \setminus \{\bm{u}_r\}$ 
into $(p-1)$ pairs $\{\bm{u}_{\ell},\bm{u}_{\ell+p}\}$, $r-p+1 \le \ell \le r-1$. By the pigeonhole principle, there is a desired pair $\bm{u}_{\ell}, \bm{u}_{\ell+p} \in A$.

By the above claim, $A$ contains two distinct vertices $\bm{u}_r,\bm{u}_s$ with $s-r \in \{p,\ldots,q-p\}$. Recall that $|\bm{u}_{r}-\zeta^{s-r} \bm{u}_{s}|< 2\sqrt{\mu}$. \cref{lem:co-degree-2} below then implies that $\bm{u}_r$ and $\bm{u}_s$ have less than $\eps  tn$ common neighbors in 
$G_0$. Consequently, vertices in $A$ have less than $\eps  tn$ common neighbors in 
$G_0$.
\end{proof}

The following lemma makes up the bulk of the proof of \cref{lem:co-degree-1}.

\begin{lemma}\label{lem:co-degree-2}
Let $j\in [t]$, and let $\bm{u},\bm{u}' \in V^{(j)}$ be two vertices such that $|\bm{u}-\zeta^{\ell}\bm{u}'|<2\sqrt{\mu}$ for some $\ell\in \{p,\ldots,q-p\}$. Then 
$\bm{u}$ and $\bm{u}'$ have less than $\eps  tn$ common neighbors in 
$G_0$.
\end{lemma}
\begin{proof}
Suppose to the contrary that $\bm{u}$ and $\bm{u}'$ have at least $\eps  tn$ common neighbors in $G_0$. For $\bm{x}\in \bigcup_{j \in [t]}S_j$, set
\[
J(\bm{x}):=\{\bm{v}\in V: |\Im(\zeta^s\langle \bm{v},\bm{x} \rangle)|\ge 3\mu \enskip \text{for all } s \in \Z/q\Z\}.
\]

\begin{claim}\label{claim:negligible-region}
For all $\bm{x}\in  \bigcup_{j \in [t]}S_j$, we have $|\overline{J(\bm{x})}| \le \frac16\eps tn$, where $\overline{J(\bm{x})}:=V\setminus J(\bm{x})$.    
\end{claim}
\begin{poc}
For $h\in [t]$, let $J^{(h)}=J(\bm{x})\cap V^{(h)}$. Then $|\overline{J(\bm{x})}|=|V^{(1)}\setminus J^{(1)}|+\cdots+|V^{(t)}\setminus J^{(t)}|$. Hence, in order to prove the claim, it suffices to show $|V^{(h)}\setminus J^{(h)}| \le \frac16\eps  n$
for all $h\in [t]$. Fix an arbitrary $h\in [t]$. Note that $V^{(h)}\setminus J^{(h)}$ consists of vertices $\bm{v}_m^{(h)}$ such that $|\Im
(\zeta^s\langle \bm{v}_m^{(h)}, \bm{x}\rangle)|< 3\mu$ for some $s\in \Z/q\Z$.
To bound the number of such vertices, define $L= \bigcup \{D_m^{(h)}: |\Im
\langle \bm{v}_m^{(h)}, \bm{x}\rangle|< 3\mu \}$.
Since each domain $D_m^{(h)}\ni \bm{v}_m^{(h)}$ has diameter at most $\mu/2$, 
$L$ is contained in the set $\{\bm{y} \in S_h : |\Im \langle \bm{y}, \bm{x}\rangle|< (3+\frac{1}{2})\mu\}$, which by \cref{lem:area_of_band} has measure at most $11\eps$. It follows that the number of vertices $\bm{v}_m^{(h)}$ with $|\Im \langle \bm{v}_m^{(h)}, \bm{x}\rangle|< 3\mu$ is $\lambda(L)n \le 11\eta n$. For each $s\in \Z/q\Z$, since $\zeta^s\langle \bm{v}_m^{(h)}, \bm{x}\rangle=\langle \bm{v}_m^{(h)}, \zeta^{-s}\bm{x}\rangle$ and $\zeta^{-s}\bm{x}\in \bigcup_{j \in [t]}S_j$, the previous argument shows that there are at most $11\eta n$ vertices $\bm{v}_m^{(h)}$ satisfying $|\Im
(\zeta^s\langle \bm{v}_m^{(h)}, \bm{x}\rangle)|< 3\mu$. Thus, by the union bound, we get $|V^{(h)}\setminus J^{(h)}| \le 11q\eta n \le \frac16 \eps  n$, as desired. 
\end{poc}

Let $W=(N_{G_0}(\bm{u})\cap  J(\bm{u})) \cap (N_{G_0}(\bm{u}')\cap J(\bm{u}'))$. From \cref{claim:negligible-region} and our hypothesis, we see that 
\[
|W\setminus V^{(j)}|\ge |N_{G_0}(\bm{u})\cap  N_{G_0}(\bm{u}')|-|\overline{J(\bm{u})}|-|\overline{J(\bm{u}')}|-|V^{(j)}| \ge \eps tn-(\tfrac13 \eps  tn+n) \ge \tfrac12 \eps  tn
\]
assuming $\eps  t \ge 6$.
Thus, by \cref{lem: exist q vertices}, there exist $h \in [t]\setminus \{j\}$ and $q$ vertices $\bm{w}_0,\ldots, \bm{w}_{q-1}\in W \cap V^{(h)}$ such that for all $r,s \in \Z/q\Z$, $|\bm{w}_r-\zeta^{s-r} \bm{w}_s| < 2\sqrt{\mu}$.

\begin{claim}\label{claim}
$|\langle \frac{1}{q}(\bm{w}_0+\cdots +\bm{w}_{q-1}), \bm{u}-\zeta^{\ell}\bm{u}'\rangle|< 4\mu$.  
\end{claim}
\begin{poc}
By the triangle inequality, 
\begin{align*}
    |\bm{w}_0+\cdots +\bm{w}_{q-1}|&\le |(1+\zeta+\cdots+\zeta^{q-1})\bm{w}_0|+ \sum_{s=1}^{q}|\bm{w}_s-\zeta^{q-s}\bm{w}_0|\\
    &= \sum_{s=1}^{q}|\bm{w}_s-\zeta^{-s}\bm{w}_0|< 2q\sqrt{\mu},
\end{align*}
where in the second line we used the facts that $1+\zeta+\cdots+\zeta^{q-1}=0$ and $\zeta^q=1$. 
Thus, the Cauchy--Schwarz inequality gives $|\langle \frac{1}{q}(\bm{w}_0+\cdots +\bm{w}_{q-1}), \bm{u}-\zeta^{\ell}\bm{u}'\rangle|\le |\frac{1}{q}(\bm{w}_0+\cdots +\bm{w}_{q-1})|\cdot |\bm{u}-\zeta^{\ell}\bm{u}'| < 4\mu$.
\end{poc}

When $j>h$, we get the following contradiction.

\begin{claim}\label{claim: case j>h}
If $j>h$, then $|\langle \frac{1}{q}(\bm{w}_0+\cdots +\bm{w}_{q-1}), \bm{u}-\zeta^{\ell}\bm{u}'\rangle|\ge 6\mu$.  
\end{claim}
\begin{poc}
Consider an arbitrary $s\in \Z/q\Z$. As $h<j$, the cross pairs $(\bm{w}_s,\bm{u}), (\bm{w}_s,\bm{u}')\in V^{(h)}\times V^{(j)}$ are edges of $G_0$ and $p/q\le 1/2$, we have 
\begin{equation}\label{eq:arg-1}
\arg \langle \bm{w}_s,\bm{u}\rangle,\arg \langle \bm{w}_s,\bm{u}'\rangle\in [0, \tfrac{2\pi p}{q}] \subseteq [0,\pi].    
\end{equation}
This, together with the fact that $\bm{w}_s\in J(\bm{u})$, shows
\[
\Im \langle \bm{w}_s,\bm{u}\rangle=|\Im \langle \bm{w}_s,\bm{u}\rangle| \ge 3\mu.
\]
Suppose that $p \le \ell \le q/2$. Then \eqref{eq:arg-1} implies $\arg(-\zeta^{-\ell}\langle \bm{w}_s,\bm{u}'\rangle) \in [\pi-\frac{2\pi\ell}{q},\pi+\frac{2\pi p}{q}-\frac{2\pi\ell}{q}] \subseteq [0,\pi]$.
Hence  
\begin{equation*}
\Im \langle \bm{w}_s,-\zeta^{\ell}\bm{u}'\rangle\overset{\cref{fact: inner product}}{=}\Im(-\zeta^{-\ell}\langle \bm{w}_s,\bm{u}'\rangle)=|\Im(\zeta^{-\ell}\langle \bm{w}_s,\bm{u}'\rangle)| \ge 3\mu    
\end{equation*}
for $\bm{w}_s\in J(\bm{u}')$.
Therefore, in this case we have
\[
|\langle \tfrac{1}{q}(\bm{w}_0+\cdots +\bm{w}_{q-1}), \bm{u}-\zeta^{\ell}\bm{u}'\rangle| \ge \Im\, \langle \tfrac{1}{q}(\bm{w}_0+\cdots +\bm{w}_{q-1}), \bm{u}-\zeta^{\ell}\bm{u}'\rangle \ge 6\mu.
\]
Now suppose that $q/2 \le \ell \le q-p$. We learn from \eqref{eq:arg-1} that $\arg(\zeta^{-p}\langle \bm{w}_s,\bm{u}\rangle) \in [-\frac{2\pi p}{q},0] \subseteq [-\pi,0]$, and that $\arg(-\zeta^{-p-\ell}\langle \bm{w}_s,\bm{u}'\rangle) \in [\pi-\frac{2\pi (p+\ell)}{q},\pi-\frac{2\pi \ell}{q}] \subseteq [-\pi,0]$. Hence,
\begin{align*}
  \Im(\zeta^{-p}\langle \bm{w}_s,\bm{u}\rangle) &= -|\Im(\zeta^{-p}\langle \bm{w}_s,\bm{u}\rangle)| \le -3\mu, \mbox{\ and }
\\
\Im(\zeta^{-p} \langle \bm{w}_s,-\zeta^{\ell}\bm{u}'\rangle) &=\Im(-\zeta^{-p-\ell} \langle \bm{w}_s,\bm{u}'\rangle)=-|\Im(\zeta^{-p-\ell}\langle \bm{w}_s,\bm{u}'\rangle)| \le -3\mu.  
\end{align*}
Thus, in the case when $q/2\le \ell \le q-p$ we also have
\begin{align*}
|\langle \tfrac{1}{q}(\bm{w}_0+\cdots +\bm{w}_{q-1}), \bm{u}-\zeta^{\ell}\bm{u}'\rangle| &=|\zeta^{-p}\langle \tfrac{1}{q}(\bm{w}_0+\cdots +\bm{w}_{q-1}), \bm{u}-\zeta^{\ell}\bm{u}'\rangle|\\
&\ge -\Im(\zeta^{-p}\langle \tfrac{1}{q}(\bm{w}_0+\cdots +\bm{w}_{q-1}), \bm{u}-\zeta^{\ell}\bm{u}'\rangle)\ge 6\mu,
\end{align*}
completing the proof of \cref{claim: case j>h}.
\end{poc}

In the following claim, we deal with the case when $j<h$.
\begin{claim}\label{claim: case j<h}
If $j<h$, then $|\langle \frac{1}{q}(\bm{w}_0+\cdots +\bm{w}_{q-1}), \bm{u}-\zeta^{\ell}\bm{u}'\rangle|\ge 6\mu$.
\end{claim}
\begin{poc}
We have $\arg \langle \bm{w}_s,\bm{u}\rangle, \arg \langle \bm{w}_s,\bm{u}'\rangle\in [\frac{2\pi (q-p)}{q}, 2\pi]$. Thus, 
$\arg( \zeta^p\langle \bm{w}_s,\bm{u}\rangle),  \arg(\zeta^p\langle \bm{w}_s,\bm{u}'\rangle) \in [0, \frac{2\pi p}{q}]$.
From this one can show that $\arg (\zeta^p \langle \bm{w}_s,\bm{u}\rangle), \arg (\zeta^p\langle \bm{w}_s,-\zeta^{\ell} \bm{u}'\rangle) \in [0,\pi]$
when $p \le \ell \le q/2$, and that
$\arg \langle \bm{w}_s,\bm{u}\rangle, \arg \langle \bm{w}_s,-\zeta^{\ell} \bm{u}'\rangle)\in [\pi, 2\pi]$ 
when $q/2\le \ell \le q-p$.

We first consider the case when $p \le \ell \le q/2$. As $\bm{w}_s\in J(\bm{u})\cap J(\bm{u}')$, using \cref{fact: inner product} we get that $\Im(\zeta^p \langle \bm{w}_s, \bm{u}\rangle) =-\Im(\zeta^{-p}\langle \bm{u},\bm{w}_s\rangle)\ge 3\mu$ and $\Im(\zeta^p \langle \bm{w}_s, -\zeta^{\ell} \bm{u}'\rangle)\ge 3\mu$.
Hence, 
\[
|\langle \tfrac{1}{q}(\bm{w}_0+\cdots +\bm{w}_{q-1}), \bm{u}-\zeta^{\ell}\bm{u}'\rangle|\ge \Im (\zeta^p \langle \tfrac{1}{q}(\bm{w}_0+\cdots +\bm{w}_{q-1}), \bm{u}-\zeta^{\ell}\bm{u}'\rangle) \ge 6\mu.
\]
For $q/2\le \ell \le q-p$, the same argument gives that $\Im \langle \bm{w}_s, \bm{u}\rangle \le -3\mu$ and $ \Im \langle \bm{w}_s, -\zeta^{\ell} \bm{u}'\rangle\le -3\mu$, so
\[
|\langle \tfrac{1}{q}(\bm{w}_0+\cdots +\bm{w}_{q-1}), \bm{u}-\zeta^{\ell}\bm{u}'\rangle|\ge -\Im \langle \tfrac{1}{q}(\bm{w}_0+\cdots +\bm{w}_{q-1}), \bm{u}-\zeta^{\ell}\bm{u}'\rangle\ge 6\mu,
\]
which finishes the proof of \cref{claim: case j<h}.
\end{poc}
As Claims \ref{claim: case j>h} and \ref{claim: case j<h} contradict \cref{claim}, we conclude that 
$\bm{u}$ and $\bm{u}'$ have less than $\eps  tn$ common neighbors in $G_0$.
\end{proof}

\subsection{The case when \texorpdfstring{$\frac{p}{q}>\frac{1}{2}$}{dense}} 
\paragraph{Construction 2.} 
Note that $\frac{q-p}{q} < \frac{1}{2}$ as $\frac{p}{q}>\frac{1}{2}$. Let $G_0, \ldots, G_{q-1}$ be the graphs (on the same vertex set $V$) obtained by using Construction 1 with $\Big(q-p,q,\frac{\eps}{2q^2{q \choose p-1}}\Big)$ playing the role of $(p,q,\eps)$.

We will show that the complement graphs $\overline{G}_0, \ldots, \overline{G}_{q-1}$ satisfy \cref{thm:ramsey_restated}.

\begin{proof}[Proof of \cref{thm:ramsey_restated} for $\frac{p}{q}> \frac{1}{2}$]
(i) This follows immediately from the construction and part (i) of the case $p/q\le 1/2$. 

(ii) Let $f \in \Z/q\Z$ and $U \subseteq V$ be any vertex subset of size at least $\eps  tn$.  
We need to show that 
$U$ contains $q$ vertices, 
any $p+1$ of 
which have less than $\eps  tn$ common neighbors in $\overline{G}_f$.
It was shown in \cref{sec:3.1} that there exist $j \in [t]$ and $q$ vertices $ \bm{u}_{1}, \ldots, \bm{u}_{q} \in U \cap V^{(j)}$ such that every $q-p+1$ of them have less than $\frac{\eps  nt}{2q^2 {q\choose p-1}}$ common neighbors in $G_f$.
Thus, the number of vertices of $G_f$ having at least $q-p+1$ neighbors in $\{ \bm{u}_{1}, \ldots, \bm{u}_{q}\}$ is less than ${q \choose q-p+1} \cdot \frac{\eps  tn}{2q^2 {q\choose p-1}} = \frac{\eps  tn}{2q^2}$.
Let $X$ be the set of vertices of $G_f$ such that each of them is adjacent to exactly $q-p$ vertices of $\{ \bm{u}_{1}, \ldots, \bm{u}_{q}\}$. 
Then, the sum of degrees of $\bm{u}_{1}, \ldots, \bm{u}_{q}$ in $G_f$ is less than 
\[
\frac{\eps  tn}{2q^2} \cdot q + tn \cdot (q-p-1) + |X|.
\]
On the other hand, by Lemma~\ref{lem:degree} below, the sum of degrees of $ \bm{u}_{1}, \ldots, \bm{u}_{q}$ in $G_f$ is at least 
\[
q\cdot (t-1) \cdot \Big(\frac{q-p}{q}-\frac{\eps }{6q}\Big)n.
\]
Comparing these two bounds, we get $|X|>  tn -(q-p)n- \frac{\eps tn}{2q} -\frac{\eps  (t-1)n}{6} \ge tn-\eps  tn$.
It shows that all but less than $\eps  tn$ vertices of $\overline{G}_f$ have exactly $p$ neighbors in $\{\bm{u}_{1}, \ldots, \bm{u}_{q}\}$. Hence, in $\overline{G}_f$, every $p+1$ vertices of $\{\bm{u}_{1}, \ldots, \bm{u}_{q}\}$ have less than $\eps  tn$ common neighbors.
\end{proof}

To complete the proof of \cref{thm:ramsey_restated}, we only need to prove the following lemma.

\begin{lemma}\label{lem:degree}
    For any distinct $j, h\in [t]$ and $f \in \Z/q\Z$, in the graph $G_f$ every vertex of $V^{(j)}$ has at least $(\frac{q-p}{q} - \frac{\eps}{6q})n$ neighbors in $V^{(h)}$. 
\end{lemma}
\begin{proof}
    We consider only the case $h<j$ as the case $h>j$ can be argued similarly.
    Fix any vertex $\bm{u} \in V^{(j)}$. By rotating the sphere $S_h$ by the map $\bm{x} \mapsto \zeta^{f}\bm{x}$, we may assume that $f=0$. For $\bm{y}\in S_j$, let 
    $$I(\bm{y}) = \Big\{\bm{x} \in S_h : \arg\langle \bm{x}, \bm{y} \rangle \in \big[0, \tfrac{2\pi (q-p)}{q}\big]  \Big\}$$
    and 
    $$ K(\bm{y}) = \{\bm{x}\in S_h: |\Im(\zeta^s\langle \bm{x}, \bm{y} \rangle)|\le 3\mu \enskip \text{for some } s \in \Z/q\Z\}.$$
    First we prove the following claim.
    \begin{claim}\label{claim: Dm intersect}
        Let $m\in [n]$. There is no $D_m^{(h)}$ intersecting both $I(\bm{u}) \setminus K(\bm{u})$ and $\overline{I(\bm{u})} \setminus K(\bm{u})$. 
    \end{claim}
   \begin{poc} 
    Suppose to the contrary that $\bm{z}\in D_m^{(h)}\cap I(\bm{u}) \setminus K(\bm{u})$ and $\bm{z^*}\in D_m^{(h)}\cap \overline{I(\bm{u})} \setminus K(\bm{u})$. Note that the diameter of $D_m^{(h)}$ is at most $\mu/2$, we have $|\bm{z}-\bm{z^*}|\leq \mu/2$, which implies that $|\langle \bm{z},\bm{u}\rangle-\langle \bm{z^*}, \bm{u}\rangle|\leq |\bm{z}-\bm{z^*}|\cdot|\bm{u}|\leq \mu/2$ by Cauchy–Schwarz inequality. Now we show that $|\langle \bm{z},\bm{u}\rangle-\langle \bm{z^*}, \bm{u}\rangle|\geq 6\mu$ to get a contradiction.
    
    Let $\mathbb{D} = \{z \in \mathbb{C} : |z| \leq 1\}$ be the unit circle in $\mathbb{C}$ 
    and let $$\widetilde{K} = \{z \in \mathbb{D} : |\Im(\zeta^s z)|\le 3\mu \enskip \text{for some } s \in \Z/q\Z\}.$$
    Then $\mathbb{D} \setminus \widetilde{K}$ consists  
    of $2q$ connected components when $q$ is odd, and $q$ components when $q$ is even.
    We now consider the case when $q$ is odd; the case $q$ is even is similar. 
    For $\ell\in[2q]$, let $I_\ell$ be the component contained in $\{ z \in \mathbb{D} : \arg z \in [\frac{2\pi (\ell-1)}{2q}, \frac{2\pi\ell}{2q}]\}$.
    Observe that each strip $\{z \in \mathbb{D} : |\Im(\zeta^s z)| \leq 3\mu\}$, $s\in \Z/q\Z$, has width $6\mu$. We can then infer that every pair of distinct components $I_{\ell_1}$ and $I_{\ell_2}$ have distance at least $6 \mu$ as they are separated by one of those strips. This finishes the proof as $\langle \bm{z},\bm{u}\rangle \in \bigcup_{\ell \in [2(q-p)]} I_\ell$ and $\langle \bm{z^*},\bm{u}\rangle \in \bigcup_{\ell \in [2q] \setminus [2(q-p)]} I_\ell$ are in different components. 
    \end{poc}

    By \cref{claim: Dm intersect}, for every $D_m^{(h)}$ intersecting with $I(\bm{u}) \setminus K(\bm{u})$, either $\bm{v}_m^{(h)} \in N_{G_0}(\bm{u}) \cap V^{(h)}$ or $D_m^{(h)} \cap K(\bm{u}) \neq \varnothing$.
    Let $W =K(\bm{u})\cap V^{(h)}$. 
    Then by the same argument as in \cref{claim:negligible-region}, $\lambda(K(\bm{u}))\le 11q\eta $ and $|W| \leq 11 q \eta n$.
    Therefore, we have $|N_{G_0}(\bm{u}) \cap V^{(h)}| \geq \lambda(I(\bm{u}) \setminus K(\bm{u})) n  - |W| \geq \frac{q-p}{q}n - 11q\eta n-11 q \eta n  \geq (\frac{q-p}{q} - \frac{\eps}{6q})n$, 
    concluding the proof.
\end{proof}

\end{document}